\documentclass{amsart}

\usepackage{ adjustbox }
\usepackage{ amsthm }
\usepackage{ amssymb }
\usepackage{ amsmath }
\usepackage{ array }
\usepackage{ booktabs }
\usepackage{ cite }
\usepackage{ color }
\usepackage{ enumitem }
\usepackage{ latexsym }
\usepackage{ url }
\usepackage[all]{ xy }

\theoremstyle{definition}

\newtheorem{definition}{Definition}[section]

\newtheorem{example}[definition]{Example}

\newtheorem{remark}[definition]{Remark}

\theoremstyle{plain}

\newtheorem{lemma}[definition]{Lemma}
\newtheorem{proposition}[definition]{Proposition}
\newtheorem{theorem}[definition]{Theorem}
\newtheorem{corollary}[definition]{Corollary}

\makeatletter
\renewcommand{\tocsection}[3]{%
\indentlabel{\@ifnotempty{#2}{\makebox[1.50em][l]{\ignorespaces#1#2.}}}#3}
\renewcommand{\tocsubsection}[3]{%
\indentlabel{\@ifnotempty{#2}{\hspace*{1.50em}\makebox[2.25em][l]{\ignorespaces#1#2.}}}#3}
\renewcommand{\tocsubsubsection}[3]{%
\indentlabel{\@ifnotempty{#2}{\hspace*{3.75em}\makebox[3.00em][l]{\ignorespaces#1#2.}}}#3}
\makeatother

\allowdisplaybreaks

\begin{document}

\title[Linear operators on associative algebras]
{A new classification of algebraic identities for linear operators on associative algebras}

\author{Murray R. Bremner}

\address{Department of Mathematics and Statistics,
University of Saskatchewan,
Saskatoon, Canada}

\email{bremner@math.usask.ca}

\author{Hader A. Elgendy}

\address{Department of Mathematics, Faculty of Science, Damietta University, New Damietta, Egypt}

\email{haderelgendy42@hotmail.com}



\thanks{The first author was supported by the Discovery Grant \emph{Algebraic Operads} from NSERC,
the Natural Sciences and Engineering Research Council of Canada.
The second author was supported by a Postdoctoral Fellowship from
the Ministry of Higher Education of Egypt.}

\begin{abstract}
We introduce a new approach to the classification of operator identities,
based on basic concepts from the theory of algebraic operads
together with computational commutative algebra applied to determinantal ideals
of matrices over polynomial rings.
We consider operator identities of degree 2 (the number of variables in each term)
and multiplicity 1 or 2 (the number of operators in each term),
but our methods apply more generally.
Given an operator identity with indeterminate coefficients,
we use partial compositions to construct a matrix of consequences,
and then use computer algebra to determine the values of the indeterminates
for which this matrix has submaximal rank.
For multiplicity 1 we obtain six identities, including the derivation identity.
For multiplicity 2 we obtain eighteen identities and two parametrized families,
including the left and right averaging identities, the Rota-Baxter identity,
the Nijenhuis identity, and some new identities which deserve further study.
\end{abstract}

\maketitle

{\footnotesize\tableofcontents}


\section{Introduction}


\subsection{Rota's problem}

Rota \cite{Rota1995} posed the problem of classifying the algebraic identities that
can be satisfied by a linear operator on an associative algebra:
``%
{\dots}~%
[A] powerful technique for the study of an algebra
is the classification of its endomorphisms and their infinitesimal analogues, namely derivations.
{\dots}~%
[T]he best known example is the Galois theory of fields, which depends on the classification of
the automorphisms of a field, or, for fields of positive characteristic, of the derivations of the field.
{\dots}~%
[B]oth endomorphisms and derivations of an algebra are instances of linear operators
satisfying algebraic identities.
An endomorphism $E$ {\dots}~%
satisfies the identity $E(xy) = E(x)E(y)$, and
a derivation satisfies the identity $D(xy) = D(x)y + xD(y)$.
{\dots}~%
[O]ther linear operators satisfying algebraic identities may be of equal importance
in studying certain algebraic phenomena
\dots.''
He further states:
``I have posed the problem of finding all possible algebraic identities
that can be satisfied by a linear operator on an [associative] algebra.
Simple computations show that the possibilities are very few,
and the problem of classifying all such identities is very probably completely solvable.''


\subsection{Basic definitions}

An \emph{associative algebra} is a vector space $A$ with a bilinear operation
$B\colon A \times A \to A$ which is denoted by juxtaposition $(a,b) \mapsto B(a,b) = ab$
and which satisfies the relation $(ab)c = a(bc)$ for all $a, b, c \in A$.
A \emph{monomial of degree $p$} in $A$ is a product $a_1 a_2 \cdots a_p$ where the $p$ indeterminates
represent arbitrary elements of $A$.
We write $L\colon A \to A$ for an arbitrary linear operator on
(the underlying vector space of) $A$.
An \emph{operator monomial of degree $p$ and multiplicity $q$} is a monomial of degree $p$
which contains $q$ occurrences of $L$.
We assume that each monomial contains the same permutation of the indeterminates.
For the related topic of unary (or operated) semigroups, see Guo \cite{Guo2009}.

For example, there are three operator monomials of degree 2 and multiplicity 1, namely
$L(xy)$, $L(x)y$, $xL(y)$, and six operator monomials of degree 2 and multiplicity 2, namely
$L^2(xy)$, $L(L(x)y)$, $L^2(x)y$, $L(xL(y))$, $L(x)L(y)$, $xL^2(y)$.

An \emph{operator polynomial} $I$ is a linear combination of operator monomials of the same degree.
An \emph{operator identity} is an equation of the form $I_1 = I_2$ which is assumed
to hold for all values of the indeterminates.
A \emph{homogeneous} operator identity is one in which each term has the same multiplicity.

Table \ref{wellknown} displays some well-known operator identities.
For further information,
see
Gao \& Zhang \cite{GaoZhang2018},
Gao, Lei \& Zhang \cite{GaoLeiZhang2018},
Guo \cite{Guo2012},
Lei \& Guo \cite{LeiGuo2012},
Pei \& Gao \cite{PeiGuo2015}, and
Rota \cite{Rota1964}.
For early work on the classification problem for operator identities,
see Targonski \cite{Targonski1967},
Dhombres \cite{Dhombres1971},
and Freeman \cite{Freeman1972}.
For recent papers more closely related to our approach,
see Guo, Sit \& Zhang \cite{GSZ1,GSZ2} and Gao \& Guo \cite{GaoGuo2017}.

\begin{table}[ht]
$\begin{array}{l@{\qquad}l}
Homogeneous \\[2pt]
\text{Derivation} & L(xy) = L(x)y + xL(y) \\[2pt]
\text{Rota-Baxter} & L(x)L(y) = L(L(x)y) + L(xL(y)) \\[2pt]
\text{Left averaging} & L(x)L(y) = L(L(x)y) \\[2pt]
\text{Right averaging} & L(x)L(y) = L(xL(y)) \\[2pt]
\text{Nijenhuis} & L(x)L(y) = L(xL(y)) + L(L(x)y) - L^2(xy) \\[6pt]
Inhomogeneous \\[2pt]
\text{Endomorphism} & L(xy) = L(x)L(y) \\[2pt]
\text{Derivation, weight $\gamma \ne 0$} & L(xy) = L(x)y + xL(y) + \gamma L(x)L(y) \\[2pt]
\text{Rota-Baxter, weight $\gamma \ne 0$} & L(x)L(y) = L(L(x)y) + L(xL(y)) + \gamma L(xy) \\[2pt]
\text{Reynolds} & L(x)L(y) = L(L(x)y) + L(xL(y)) - L(L(x)L(y))
\end{array}$
\medskip
\caption{Homogeneous and inhomogeneous operator identities}
\label{wellknown}
\end{table}


\subsection{Results of this paper}

We introduce a new approach to the classification of operator identities,
using combinatorics of operator monomials,
basic concepts from the theory of algebraic operads,
and computational commutative algebra.

We start with an operator identity with indeterminate coefficients.
We apply the operadic concept of partial composition
to produce a set of new identities for which both the degree and multiplicity
increase by 1.
We call these new identities the \emph{consequences} of the original identity.
We store the coefficient vectors of these consequences in a matrix over
the polynomial ring generated by the indeterminates.
We call this the \emph{matrix of consequences} of the original identity.
We introduce what we call the \emph{rank principle}, namely that significant operator identities
correspond those values of the coefficients which produce submaximal ranks
of the matrix of consequences.
We use computational commutative algebra to determine how the rank depends on the coefficients.
This approach provides a justification of Rota's claim that
the possible identities are very few.
Our main results are Theorems \ref{theorem21} and \ref{theorem22}.
For degree 2 and multiplicity 1 we obtain six identities, including the derivation identity.
For degree 2 and multiplicity 2 we obtain eighteen identities and two parametrized families,
including the left and right averaging identities, the Rota-Baxter identity,
the Nijenhuis identity, and some new identities  which deserve further study.

We consider only operator identities of degree 2
which are homogeneous of multiplicity 1 or 2.
We assume the base field has characteristic 0.
However, our methods apply in principle to more general problems, such as
inhomogeneous operator identities,
identities whose terms contain different permutations of the arguments,
identities involving more than one linear operator,
two or more simultaneous identities for the same operator,
identities for operators on nonassociative algebras,
identities for algebras where the operation is $n$-ary for $n \ge 3$,
and
identities over fields of positive characteristic.


\section{Narayana numbers and operator monomials}

\begin{definition}
For $i \ge j \ge 1$ the \emph{Narayana numbers} are defined by
\[
N(i,j) = \frac1i \binom{i}{j} \binom{i}{j{-}1}.
\]
\end{definition}

Narayana numbers have many different combinatorial interpretations;
see the OEIS \cite[sequence A001263]{OEIS} and
Petersen \cite[Chapter 2, especially \S2.4 on Dyck paths and Dyck words]{P}.

\begin{definition}
For $i \ge 1$, consider strings of length $2i$ consisting of
$i$ left parentheses and $i$ right parentheses.
Such a string is \emph{balanced} if in every initial substring
the number of right parentheses is no greater than the number of left parentheses.
A \emph{nesting} is a substring of the form $()$.
\end{definition}

\begin{lemma}
\label{narayanalemma}
$N(i,j)$ is the number of balanced strings of length $2i$ which contain $j$ nestings.
\end{lemma}

\begin{definition}
Let $\mathcal{M}(p,q)$ be the set of all operator monomials of
degree $p$ and multiplicity $q$.
Let $\mathcal{O}(p,q)$ be the vector space with basis $\mathcal{M}(p,q)$.
An \emph{operator identity of degree $p$ and multiplicity $q$} is
any (nonzero) element of $\mathcal{O}(p,q)$.
\end{definition}

\begin{lemma}
\label{dimensionlemma}
For $p \ge 1$ and $q \ge 0$ we have
\[
\dim \mathcal{O}(p,q) = N(p{+}q,p) = \frac{1}{p{+}q} \binom{p{+}q}{p} \binom{p{+}q}{p{-}1}.
\]
\end{lemma}

\begin{proof}
Let $\mathcal{S}(p,q)$ be the set of all strings
containing $2(p{+}q)$ balanced pairs of parentheses, $p$ of which are nestings.
By Lemma \ref{narayanalemma} we have $| \mathcal{S}(p,q) | = N(p{+}q,p)$.
Apply the following algorithm to each string in $\mathcal{S}(p,q)$:
\begin{itemize}
\item[1)]
Replace each nesting by the argument symbol $\ast$.
\item[2)]
Insert the operator symbol $L$
immediately to the left of each remaining left parenthesis,
and leave the remaining right parentheses unchanged.
\item[3)]
Replace the argument symbols from left to right by $a_1 \cdots a_p$.
\end{itemize}
For example,
\[
(((())()())())()
\,\longmapsto\,
((({\ast}){\ast}{\ast}){\ast}){\ast}
\,\longmapsto\,
L(L(L({\ast}){\ast}{\ast}){\ast}){\ast}
\,\longmapsto\,
L(L(L(v)wx)y)z.
\]
Clearly this process gives a bijection between $\mathcal{S}(p,q)$ and $\mathcal{M}(p,q)$.
\end{proof}

\begin{definition}
\label{monomialorder}
The bijection in the proof of Lemma \ref{dimensionlemma} induces
a \emph{lexicographic order} on operator monomials $v, w \in \mathcal{M}(p,q)$.
Replace $v$ and $w$ by the corresponding strings $s(v)$ and $s(w)$ of parentheses.
Replace left and right parentheses in $s(v)$ and $s(w)$ by the letters $a$ and $b$ respectively
obtaining the words $t(v)$ and $t(w)$.
Then $v \prec w$ if and only if $t(v) \prec t(w)$ in dictionary order.
We used this order in Table \ref{p2q123}.
\end{definition}

Table \ref{p2q123} displays the bijection and the lexicographical order
for $p = 2$, $1 \le q \le 3$.

\begin{table}[ht]
$
\begin{array}{lrll}
q = 1
&\qquad  1  &\qquad  (()())  &\qquad  L(xy)  \\
&\qquad  2  &\qquad  (())()  &\qquad  L(x)y  \\
&\qquad  3  &\qquad  ()(())  &\qquad  xL(y)  \\[2mm]
q = 2
&\qquad  1  &\qquad  ((()()))  &\qquad  L(L(xy)) = L^2(xy) \\
&\qquad  2  &\qquad  ((())())  &\qquad  L(L(x)y)  \\
&\qquad  3  &\qquad  ((()))()  &\qquad  L(L(x))y = L^2(x)y \\
&\qquad  4  &\qquad  (()(()))  &\qquad  L(xL(y))  \\
&\qquad  5  &\qquad  (())(())  &\qquad  L(x)L(y)  \\
&\qquad  6  &\qquad  ()((()))  &\qquad  xL(L(y)) = xL^2(y) \\[2mm]
q = 3
&\qquad  1  &\qquad  (((()())))  &\qquad  L(L(L(xy))) = L^3(xy) \\
&\qquad  2  &\qquad  (((())()))  &\qquad  L(L(L(x)y)) = L^2(L(x)y) \\
&\qquad  3  &\qquad  (((()))())  &\qquad  L(L(L(x))y) = L(L^2(x)y) \\
&\qquad  4  &\qquad  (((())))()  &\qquad  L(L(L(x)))y = L^3(x)y \\
&\qquad  5  &\qquad  ((()(())))  &\qquad  L(L(xL(y))) = L^2(xL(y)) \\
&\qquad  6  &\qquad  ((())(()))  &\qquad  L(L(x)L(y))  \\
&\qquad  7  &\qquad  ((()))(())  &\qquad  L(L(x))L(y) = L^2(x)L(y) \\
&\qquad  8  &\qquad  (()((())))  &\qquad  L(xL(L(y))) = L(xL^2(y)) \\
&\qquad  9  &\qquad  (())((()))  &\qquad  L(x)L(L(y)) = L(x)L^2(y) \\
&\qquad 10  &\qquad  ()(((())))  &\qquad  xL(L(L(y))) = xL^3(y) \\
\end{array}
$
\bigskip
\caption{Operator monomials of degree 2 and multiplicities 1, 2, 3.}
\label{p2q123}
\end{table}


\section{Partial compositions in algebraic operads}

We refer to the monographs \cite{BD-book,LV-book,MSS-book}
for detailed expositions of the theory of algebraic operads.
In particular, \cite[Chapters 7-10]{BD-book} discusses applications of
computational commutative algebra to classification of operads,
and similar methods will be applied later in this paper.

\begin{definition}
Let $m \in \mathcal{M}(p,q)$.
For $1 \le i \le p$, we write $m \circ_i B$ for
the \emph{$i$-th partial composition} of $m$ with $B$:
introduce a new indeterminate $a_{p+1}$,
replace the $i$-th argument $a_i$ in $m$ by $a_i a_{p+1}$,
and then replace the subscripts of the indeterminates
by the identity permutation.
Clearly $m \circ_i B \in \mathcal{M}(p{+}1,q)$.
\end{definition}

\begin{example}
If $m = L(L(x)yL(z))$ then
\begin{align*}
m \circ_1 B &= L(L(wx)yL(z)), \\
m \circ_2 B &= L(L(w)xyL(z)), \\
m \circ_3 B &= L(L(w)xL(yz)).
\end{align*}
\end{example}

\begin{definition}
Let $m \in \mathcal{M}(p,q)$.
For $1 \le j \le 2$, we write $B \circ_j m$ for
the \emph{$j$-th partial composition} of $B$ with $m$:
introduce a new indeterminate $a_{p+1}$,
multiply $m$ on the right ($j = 1$) or the left ($j = 2$) by $a_{p+1}$,
and then replace the subscripts of the indeterminates
by the identity permutation.
Clearly $B \circ_j m \in \mathcal{M}(p{+}1,q)$.
\end{definition}

\begin{example}
If $m = L(L(x)yL(z))$ then
\begin{align*}
B \circ_1 m &= L(L(w)xL(y))z, \\
B \circ_2 m &= wL(L(x)yL(z)).
\end{align*}
\end{example}

\begin{definition}
Let $m \in \mathcal{M}(p,q)$.
For $1 \le i \le p$, we write $m \circ_i L$ for
the \emph{$i$-th partial composition} of $m$ with $L$:
apply the linear operator $L$ to the $i$-th argument $a_i$ of $m$.
Clearly $m \circ_i L \in \mathcal{M}(p,q{+}1)$.
\end{definition}

\begin{example}
If $m = L(L(x)yL(z))$ then
\begin{align*}
m \circ_1 L &= L(L^2(x)yL(z)), \\
m \circ_2 L &= L(L(x)L(y)L(z)), \\
m \circ_3 L &= L(L(x)yL^2(z)).
\end{align*}
\end{example}

\begin{definition}
Let $m \in \mathcal{M}(p,q)$.
We write $L \circ m$ for the \emph{partial composition} of $L$ with $m$:
apply the linear operator $L$ to $m$.
Clearly $L \circ m \in \mathcal{M}(p,q{+}1)$.
\end{definition}

\begin{example}
If $m = L(L(x)yL(z))$ then
\[
L \circ m = L^2(L(w)xL(y))z.
\]
\end{example}

\begin{remark}
Partial composition extends linearly from $\mathcal{M}(p,q)$ to $\mathcal{O}(p,q)$.
We obtain $p{+}2$ linear maps from $\mathcal{O}(p,q)$ to $\mathcal{O}(p{+}1,q)$:
\[
- \circ_i B \quad (1 \le i \le p),
\qquad
B \circ_j - \quad (1 \le j \le 2),
\]
and $p{+}1$ linear maps from $\mathcal{O}(p,q)$ to $\mathcal{O}(p,q{+}1)$:
\[
- \circ_i L \quad (1 \le i \le p),
\qquad
L \circ -.
\]
\end{remark}

\begin{definition}
Let $R \in \mathcal{O}(p,q)$.
By a \emph{consequence} of $R$ we mean the new operator identity of
higher degree and/or multiplicity obtained by applying to $R$
any sequence of partial compositions with $B$ and/or $L$.
\end{definition}


\section{Rota's problem in terms of linear algebra}

Consider an operator identity $R \in \mathcal{O}(p,q)$.
We want to understand how $R$ interacts with
the associative product and the linear operator.
Hence we consider the consequences of $R$ in $\mathcal{O}(p{+}1,q{+}1)$.
We obtain consequences of degree $p{+}1$ and multiplicity $q{+}1$ in two ways:
by increasing the degree first and then the multiplicity,
or the multiplicity first and then the degree.

\begin{lemma}
\label{case1}
Assume $R \in \mathcal{O}(p,q)$.
Increasing $p$ first and then $q$,
we obtain $(p{+}2)^2$ consequences in $\mathcal{O}(p{+}1,q{+}1)$
where $1 \le i \le p$, $1 \le j \le 2$, and $1 \le k \le p{+}1$:
\[
( R \circ_i B ) \circ_k L, \qquad
( B \circ_j R ) \circ_k L, \qquad
L \circ ( R \circ_i B ), \qquad
L \circ ( B \circ_j R ).
\]
\end{lemma}

\begin{lemma}
\label{case2}
Assume $R \in \mathcal{O}(p,q)$.
Increasing $q$ first and then $p$,
we obtain $(p{+}1)(p{+}2)$ consequences in $\mathcal{O}(p{+}1,q{+}1)$
where $1 \le i, k \le p$ and $1 \le j \le 2$:
\[
( R \circ_i L ) \circ_k B, \qquad
( L \circ R ) \circ_k B, \qquad
B \circ_j ( R \circ_i L ), \qquad
B \circ_j ( L \circ R ).
\]
\end{lemma}

\begin{corollary}
If $R \in \mathcal{O}(p,q)$ then there are $(p{+}2)(2p{+}3)$
consequences of $R$ in $\mathcal{O}(p{+}1,q{+}1)$.
These consequences are not necessarily distinct.
\end{corollary}

\begin{definition}
Let $R \in \mathcal{O}(p,q)$.
We write $M(R)$ for the \emph{matrix of consequences} of $R$,
which has size
$(p{+}2)(2p{+}3) \,\times\, N(p{+}q{+}2,p{+}1)$.
The $(i,j)$ entry of $M(R)$ is the
coefficient of the $j$-th monomial in $\mathcal{M}(p{+}1,q{+}1)$
in the $i$-th consequence of $R$ from Lemmas \ref{case1} and \ref{case2}.
If $R$ has indeterminate coefficients then
$M(R)$ is a matrix over the polynomial ring in $N(p{+}q,p)$ variables.
\end{definition}

We now reformulate Rota's problem in terms of linear algebra over polynomial rings.
Let $R$ be an operator identity of degree $p$ and multiplicity $q$
with indeterminate coefficients.
Our goal is to understand how the rank of $M(R)$ depends
on the values of the indeterminates.

In this paper we consider the cases $p = 2$ and $q = 1, 2$.
We find that very few values of the rank of $M(R)$ are possible,
and that each possible submaximal rank corresponds to a small set of coefficient vectors.
This is how we obtain our new classification of operator identities.


\section{Degree 2, multiplicity 1}
\label{degtwomulone}


\subsection{Constructing the matrix of consequences}
\label{construct21}

Since we consider only operator monomials with the identity permutation of the arguments,
it is unambiguous to use the argument symbol $\ast$.
Hence the general $R \in \mathcal{O}(2,1)$ has the form
\[
R = a \, L({\ast}{\ast}) + b \, L({\ast}) {\ast} + c \, {\ast} L({\ast}),
\]
where $a, b, c$ are indeterminate parameters from the base field of characteristic 0.

Lemma \ref{case1} gives 16 distinct consequences obtained by increasing first the degree
and then the multiplicity:
\begin{align*}
( R \circ_1 B ) \circ_1 L &=
a \, L(L({\ast}){\ast}{\ast}) + b \, L(L({\ast}){\ast}){\ast} + c \, L({\ast}){\ast}L({\ast}),
\\
( R \circ_1 B ) \circ_2 L &=
a \, L({\ast}L({\ast}){\ast}) + b \, L({\ast}L({\ast})){\ast} + c \, {\ast}L({\ast})L({\ast}),
\\
( R \circ_1 B ) \circ_3 L &=
a \, L({\ast}{\ast}L({\ast})) + b \, L({\ast}{\ast})L({\ast}) + c \, {\ast}{\ast}L(L({\ast})),
\\
L \circ ( R \circ_1 B ) &=
a \, L(L({\ast}{\ast}{\ast})) + b \, L(L({\ast}{\ast}){\ast}) + c \, L({\ast}{\ast}L({\ast})),
\\
( R \circ_2 B ) \circ_1 L &=
a \, L(L({\ast}){\ast}{\ast}) + b \, L(L({\ast})){\ast}{\ast} + c \, L({\ast})L({\ast}{\ast}),
\\
( R \circ_2 B ) \circ_2 L &=
a \, L({\ast}L({\ast}){\ast}) + b \, L({\ast})L({\ast}){\ast} + c \, {\ast}L(L({\ast}){\ast}),
\\
( R \circ_2 B ) \circ_3 L &=
a \, L({\ast}{\ast}L({\ast})) + b \, L({\ast}){\ast}L({\ast}) + c \, {\ast}L({\ast}L({\ast})),
\\
L \circ ( R \circ_2 B ) &=
a \, L(L({\ast}{\ast}{\ast})) + b \, L(L({\ast}){\ast}{\ast}) + c \, L({\ast}L({\ast}{\ast})),
\\
( B \circ_1 R ) \circ_1 L &=
a \, L(L({\ast}){\ast}){\ast} + b \, L(L({\ast})){\ast}{\ast} + c \, L({\ast})L({\ast}){\ast},
\\
( B \circ_1 R ) \circ_2 L &=
a \, L({\ast}L({\ast})){\ast} + b \, L({\ast})L({\ast}){\ast} + c \, {\ast}L(L({\ast})){\ast},
\\
( B \circ_1 R ) \circ_3 L &=
a \, L({\ast}{\ast})L({\ast}) + b \, L({\ast}){\ast}L({\ast}) + c \, {\ast}L({\ast})L({\ast}),
\\
L \circ ( B \circ_1 R ) &=
a \, L(L({\ast}{\ast}){\ast}) + b \, L(L({\ast}){\ast}{\ast}) + c \, L({\ast}L({\ast}){\ast}),
\\
( B \circ_2 R ) \circ_1 L &=
a \, L({\ast})L({\ast}{\ast}) + b \, L({\ast})L({\ast}){\ast} + c \, L({\ast}){\ast}L({\ast}),
\\
( B \circ_2 R ) \circ_2 L &=
a \, {\ast}L(L({\ast}){\ast}) + b \, {\ast}L(L({\ast})){\ast} + c \, {\ast}L({\ast})L({\ast}),
\\
( B \circ_2 R ) \circ_3 L &=
a \, {\ast}L({\ast}L({\ast})) + b \, {\ast}L({\ast})L({\ast}) + c \, {\ast}{\ast}L(L({\ast})),
\\
L \circ ( B \circ_2 R ) &=
a \, L({\ast}L({\ast}{\ast})) + b \, L({\ast}L({\ast}){\ast}) + c \, L({\ast}{\ast}L({\ast})).
\end{align*}
Lemma \ref{case2} gives 12 distinct consequences obtained by increasing first the multiplicity
and then the degree, but 8 have already been obtained from Lemma \ref{case1}:
\begin{align*}
( R \circ_1 L ) \circ_1 B &=
a \, L(L({\ast}{\ast}){\ast}) + b \, L(L({\ast}{\ast})){\ast} + c \, L({\ast}{\ast})L({\ast}),
\\
( R \circ_1 L ) \circ_2 B &=
a \, L(L({\ast}){\ast}{\ast}) + b \, L(L({\ast})){\ast}{\ast} + c \, L({\ast})L({\ast}{\ast}) =
( R \circ_2 B ) \circ_1 L,
\\
B \circ_1 ( R \circ_1 L ) &=
a \, L(L({\ast}){\ast}){\ast} + b \, L(L({\ast})){\ast}{\ast} + c \, L({\ast})L({\ast}){\ast} =
( B \circ_1 R ) \circ_1 L,
\\
B \circ_2 ( R \circ_1 L ) &=
a \, {\ast}L(L({\ast}){\ast}) + b \, {\ast}L(L({\ast})){\ast} + c \, {\ast}L({\ast})L({\ast}) =
( B \circ_2 R ) \circ_2 L,
\\
( R \circ_2 L ) \circ_1 B &=
a \, L({\ast}{\ast}L({\ast})) + b \, L({\ast}{\ast})L({\ast}) + c \, {\ast}{\ast}L(L({\ast})) =
( R \circ_1 B ) \circ_3 L,
\\
( R \circ_2 L ) \circ_2 B &=
a \, L({\ast}L({\ast}{\ast})) + b \, L({\ast})L({\ast}{\ast}) + c \, {\ast}L(L({\ast}{\ast})),
\\
B \circ_1 ( R \circ_2 L ) &=
a \, L({\ast}L({\ast})){\ast} + b \, L({\ast})L({\ast}){\ast} + c \, {\ast}L(L({\ast})){\ast} =
( B \circ_1 R ) \circ_2 L,
\\
B \circ_2 ( R \circ_2 L ) &=
a \, {\ast}L({\ast}L({\ast})) + b \, {\ast}L({\ast})L({\ast}) + c \, {\ast}{\ast}L(L({\ast})) =
( B \circ_2 R ) \circ_3 L,
\\
( L \circ R ) \circ_1 B &=
a \, L(L({\ast}{\ast}{\ast})) + b \, L(L({\ast}{\ast}){\ast}) + c \, L({\ast}{\ast}L({\ast})) =
L \circ ( R \circ_1 B ),
\\
( L \circ R ) \circ_2 B &=
a \, L(L({\ast}{\ast}{\ast})) + b \, L(L({\ast}){\ast}{\ast}) + c \, L({\ast}L({\ast}{\ast})) =
L \circ ( R \circ_2 B ),
\\
B \circ_1 ( L \circ R ) &=
a \, L(L({\ast}{\ast})){\ast} + b \, L(L({\ast}){\ast}){\ast} + c \, L({\ast}L({\ast})){\ast},
\\
B \circ_2 ( L \circ R ) &=
a \, {\ast}L(L({\ast}{\ast})) + b \, {\ast}L(L({\ast}){\ast}) + c \, {\ast}L({\ast}L({\ast})).
\end{align*}
Thus for $p = 2$ we have altogether 20 distinct consequences.

The monomials in the consequences belong to the ordered monomial basis
of $\mathcal{O}(3,2)$ given in Table \ref{basis32}.
The matrix of consequences $M(R)$ has size $28 \times 20$, but 8 rows are redundant;
removing these rows gives the $20 \times 20$ matrix (also called $M(R)$)
in Table \ref{conmat21} (with dot for zero).
Our goal is to understand how the rank of this matrix depends on the values of
the parameters $a, b, c$.

Without loss of generality, we may scale $R$ so that its first nonzero coefficient is 1.
Thus we may split the problem into three mutually exclusive cases.

\begin{table}[ht]
$
\begin{array}{rll@{\qquad\qquad}rll}
 1  &\;\;\;  ((()()()))  &\;\;\;  L(L({\ast}{\ast}{\ast}))  &
 2  &\;\;\;  ((()())())  &\;\;\;  L(L({\ast}{\ast}){\ast})  \\
 3  &\;\;\;  ((()()))()  &\;\;\;  L(L({\ast}{\ast})){\ast}  &
 4  &\;\;\;  ((())()())  &\;\;\;  L(L({\ast}){\ast}{\ast})  \\
 5  &\;\;\;  ((())())()  &\;\;\;  L(L({\ast}){\ast}){\ast}  &
 6  &\;\;\;  ((()))()()  &\;\;\;  L(L({\ast})){\ast}{\ast}  \\
 7  &\;\;\;  (()(()()))  &\;\;\;  L({\ast}L({\ast}{\ast}))  &
 8  &\;\;\;  (()(())())  &\;\;\;  L({\ast}L({\ast}){\ast})  \\
 9  &\;\;\;  (()(()))()  &\;\;\;  L({\ast}L({\ast})){\ast}  &
10  &\;\;\;  (()()(()))  &\;\;\;  L({\ast}{\ast}L({\ast}))  \\
11  &\;\;\;  (()())(())  &\;\;\;  L({\ast}{\ast})L({\ast})  &
12  &\;\;\;  (())(()())  &\;\;\;  L({\ast})L({\ast}{\ast})  \\
13  &\;\;\;  (())(())()  &\;\;\;  L({\ast})L({\ast}){\ast}  &
14  &\;\;\;  (())()(())  &\;\;\;  L({\ast}){\ast}L({\ast})  \\
15  &\;\;\;  ()((()()))  &\;\;\;  {\ast}L(L({\ast}{\ast}))  &
16  &\;\;\;  ()((())())  &\;\;\;  {\ast}L(L({\ast}){\ast})  \\
17  &\;\;\;  ()((()))()  &\;\;\;  {\ast}L(L({\ast})){\ast}  &
18  &\;\;\;  ()(()(()))  &\;\;\;  {\ast}L({\ast}L({\ast}))  \\
19  &\;\;\;  ()(())(())  &\;\;\;  {\ast}L({\ast})L({\ast})  &
20  &\;\;\;  ()()((()))  &\;\;\;  {\ast}{\ast}L(L({\ast}))
\end{array}
$
\bigskip
\caption{Ordered monomial basis of $\mathcal{O}(3,2)$}
\label{basis32}
\end{table}

\begin{table}[ht]
$\left[
\begin{array}{
c@{\quad}c@{\quad}c@{\quad}c@{\quad}c@{\quad}c@{\quad}c@{\quad}c@{\quad}c@{\quad}c@{\quad}c@{\quad}c@{\quad}c@{\quad}c@{\quad}
c@{\quad}c@{\quad}c@{\quad}c@{\quad}c@{\quad}c@{\quad}c@{\quad}c@{\quad}c@{\quad}c@{\quad}c@{\quad}c@{\quad}c@{\quad}c
}
. & . & . & a & b & . & . & . & . & . & . & . & . & c & . & . & . & . & . & . \\
. & . & . & . & . & . & . & a & b & . & . & . & . & . & . & . & . & . & c & . \\
. & . & . & . & . & . & . & . & . & a & b & . & . & . & . & . & . & . & . & c \\
a & b & . & . & . & . & . & . & . & c & . & . & . & . & . & . & . & . & . & . \\
. & . & . & a & . & b & . & . & . & . & . & c & . & . & . & . & . & . & . & . \\
. & . & . & . & . & . & . & a & . & . & . & . & b & . & . & c & . & . & . & . \\
. & . & . & . & . & . & . & . & . & a & . & . & . & b & . & . & . & c & . & . \\
a & . & . & b & . & . & c & . & . & . & . & . & . & . & . & . & . & . & . & . \\
. & . & . & . & a & b & . & . & . & . & . & . & c & . & . & . & . & . & . & . \\
. & . & . & . & . & . & . & . & a & . & . & . & b & . & . & . & c & . & . & . \\
. & . & . & . & . & . & . & . & . & . & a & . & . & b & . & . & . & . & c & . \\
. & a & . & b & . & . & . & c & . & . & . & . & . & . & . & . & . & . & . & . \\
. & . & . & . & . & . & . & . & . & . & . & a & b & c & . & . & . & . & . & . \\
. & . & . & . & . & . & . & . & . & . & . & . & . & . & . & a & b & . & c & . \\
. & . & . & . & . & . & . & . & . & . & . & . & . & . & . & . & . & a & b & c \\
. & . & . & . & . & . & a & b & . & c & . & . & . & . & . & . & . & . & . & . \\
. & a & b & . & . & . & . & . & . & . & c & . & . & . & . & . & . & . & . & . \\
. & . & . & . & . & . & a & . & . & . & . & b & . & . & c & . & . & . & . & . \\
. & . & a & . & b & . & . & . & c & . & . & . & . & . & . & . & . & . & . & . \\
. & . & . & . & . & . & . & . & . & . & . & . & . & . & a & b & . & c & . & .
\end{array}
\right]$
\bigskip
\caption{The matrix $M(R)$ of consequences for general $R \in \mathcal{O}(2,1)$}
\label{conmat21}
\end{table}


\subsection{Case 1: $a = 1$ and $b, c$ are free}

We substitute $a = 1$ into the matrix $M(R)$ of Table \ref{conmat21}.
We compute the partial Smith form of the resulting matrix
(see \cite[\S8.4]{BD-book} for a detailed discussion of the PSF).
We repeatedly use elementary row and column operations
to swap 1s onto the diagonal and then use these diagonal 1s to eliminate
the nonzero elements in the corresponding row and column.
We obtain
\[
M(R)
\xrightarrow{\quad\text{PSF}\quad}
\left[
\begin{array}{cc}
I_{14} &\!\! O_{14,6} \\
O_{6,14} &\!\! B
\end{array}
\right]
\]
where $I$ and $O$ denote the identity and zero matrices of the given sizes, and
\[
B =
\left[
\begin{array}{cccccc}
b^4{+}b^3 & b^3 c{+}b^2 c & . & {-}b^2 c{+}b c^2 & {-}b c^3{-}b c^2 & {-}c^4{-}c^3
\\
. & b^2{+}b & . & . & {-}c^2{-}c & .
\\
. & . & . & . & . & .
\\
. & {-}b^3{-}b^2 & . & . & b c^2{+}b c & .
\\
. & . & . & {-}b^2{-}b & . & c^2{+}c
\\
b^2{+}b & . & . & {-}c^2{-}c & . & .
\end{array}
\right]
\]
After performing a few elementary row operations on $B$,
and deleting the resulting zero rows, we obtain the matrix
\[
B' =
\left[
\begin{array}{cccccc}
b^2{+}b & . & . & -c^2{-}c & . & . \\
. & b^2{+}b & . & . & -c^2{-}c & . \\
. & . & . & b^2{+}b & . & -c^2{-}c
\end{array}
\right]
\]
Hence $M(R)$ has minimal rank 14 and maximal rank 17.
The rank is 14 if and only if $b, c \in \{ 0, -1 \}$;
otherwise the rank is 17.
Rank 14 occurs only for
\[
(a,b,c) = (1,0,0), \; (1,0,-1), \; (1,-1,0), \; (1,-1,-1).
\]
These values of the coefficients correspond to the operator identities
\[
L(xy) = 0,
\quad\;
L(xy) = xL(y),
\quad\;
L(xy) = L(x)y,
\quad\;
L(xy) = L(x)y + xL(y).
\]


\subsection{Case 2: $a = 0$, $b = 1$ and $c$ is free}

In this case we have only one indeterminate so we can compute the ordinary Smith form.
After deleting the resulting zero rows and zero columns, we obtain
\[
\left[
\begin{array}{cc}
I_{14} & O_{14,3} \\
O_{3,14} & B
\end{array}
\right]
\quad\quad\quad
B =
\left[
\begin{array}{ccc}
c & . & .   \\
. & c & .   \\
. & . & c^2
\end{array}
\right]
\]
Hence $M(R)$ has minimal rank 14 and maximal rank 17.
The rank is 14 if and only if $c = 0$, otherwise the rank is 17.
Rank 14 occurs only for $(a,b,c) = (0,1,0)$
which corresponds to the operator identity
$L(x)y = 0$.


\subsection{Case 3: $a = b = 0$ and $c = 1$}

In this case we have no indeterminates so we can compute the ordinary row canonical form.
After deleting the resulting zero rows and zero columns,
we obtain $I_{14}$.
Hence $M(R)$ has rank 14 and $(a,b,c) = (0,0,1)$ corresponds to the operator identity
$xL(y) = 0$.

The preceding calculations prove the following theorem,
which provides our new classification of operator identities of degree 2 and multiplicity 1.

\begin{theorem}
\label{theorem21}
Consider the (nonzero) operator identity
\[
R = a \, L(xy) + b \, L(x)y + c \, xL(y).
\]
The matrix of consequences $M(R)$ has rank 14 or 17.
Rank 14 occurs if and only if the values of the coefficients
(up to nonzero scalar multiples)
correspond to one of the following six operator identities:
\begin{align*}
&
L(xy) = 0,
\qquad
L(xy) = xL(y),
\qquad
L(xy) = L(x)y,
\\
&
L(xy) = L(x)y + xL(y),
\qquad
L(x)y = 0,
\qquad
xL(y) = 0.
\end{align*}
\end{theorem}


\section{Degree 2, multiplicity 2}


\subsection{Constructing the matrix of consequences}

The general $R \in \mathcal{O}(2,2)$ has the form
\[
R
=
a L^2(xy) +
b L(L(x)y) +
c L^2(x)y +
d L(xL(y)) +
e L(x)L(y) +
f xL^2(y),
\]
where $a, b, c, d, e, f$ are indeterminate parameters from the base field of characteristic 0.
Without loss of generality, we may scale $R$ so that its first nonzero coefficient is 1,
giving six mutually exclusive cases.

The 20 distinct consequences of $R$ are formally the same (in terms of partial compositions)
as in \S\ref{construct21}.
But of course the expressions of the consequences as linear combinations of basis monomials
are different.
For example,
\begin{align*}
( R \circ_1 M ) \circ_1 L
&=
a L(L(L({\ast}){\ast}{\ast})) +
b L(L(L({\ast}){\ast}){\ast}) +
c L(L(L({\ast}){\ast})){\ast}
\\
&\quad
{} +
d L(L({\ast}){\ast}L({\ast})) +
e L(L({\ast}){\ast})L({\ast}) +
f L({\ast}){\ast}L(L({\ast})),
\\
( R \circ_1 L ) \circ_1 M
&=
a L(L(L({\ast}{\ast}){\ast})) +
b L(L(L({\ast}{\ast})){\ast}) +
c L(L(L({\ast}{\ast}))){\ast}
\\
&\quad
{} +
d L(L({\ast}{\ast})L({\ast})) +
e L(L({\ast}{\ast}))L({\ast}) +
f L({\ast}{\ast})L(L({\ast})).
\end{align*}
See Table \ref{basis33} for the ordered monomial basis of $\mathcal{O}(3,3)$.


\begin{table}[ht]
$
\begin{array}{rll@{\qquad\qquad}rll}
 1  &\;\;\;  (((()()())))  &\;\;\;  L(L(L({\ast}{\ast}{\ast})))  &
 2  &\;\;\;  (((()())()))  &\;\;\;  L(L(L({\ast}{\ast}){\ast}))  \\
 3  &\;\;\;  (((()()))())  &\;\;\;  L(L(L({\ast}{\ast})){\ast})  &
 4  &\;\;\;  (((()())))()  &\;\;\;  L(L(L({\ast}{\ast}))){\ast}  \\
 5  &\;\;\;  (((())()()))  &\;\;\;  L(L(L({\ast}){\ast}{\ast}))  &
 6  &\;\;\;  (((())())())  &\;\;\;  L(L(L({\ast}){\ast}){\ast})  \\
 7  &\;\;\;  (((())()))()  &\;\;\;  L(L(L({\ast}){\ast})){\ast}  &
 8  &\;\;\;  (((()))()())  &\;\;\;  L(L(L({\ast})){\ast}{\ast})  \\
 9  &\;\;\;  (((()))())()  &\;\;\;  L(L(L({\ast})){\ast}){\ast}  &
10  &\;\;\;  (((())))()()  &\;\;\;  L(L(L({\ast}))){\ast}{\ast}  \\
11  &\;\;\;  ((()(()())))  &\;\;\;  L(L({\ast}L({\ast}{\ast})))  &
12  &\;\;\;  ((()(())()))  &\;\;\;  L(L({\ast}L({\ast}){\ast}))  \\
13  &\;\;\;  ((()(()))())  &\;\;\;  L(L({\ast}L({\ast})){\ast})  &
14  &\;\;\;  ((()(())))()  &\;\;\;  L(L({\ast}L({\ast}))){\ast}  \\
15  &\;\;\;  ((()()(())))  &\;\;\;  L(L({\ast}{\ast}L({\ast})))  &
16  &\;\;\;  ((()())(()))  &\;\;\;  L(L({\ast}{\ast})L({\ast}))  \\
17  &\;\;\;  ((()()))(())  &\;\;\;  L(L({\ast}{\ast}))L({\ast})  &
18  &\;\;\;  ((())(()()))  &\;\;\;  L(L({\ast})L({\ast}{\ast}))  \\
19  &\;\;\;  ((())(())())  &\;\;\;  L(L({\ast})L({\ast}){\ast})  &
20  &\;\;\;  ((())(()))()  &\;\;\;  L(L({\ast})L({\ast})){\ast}  \\
21  &\;\;\;  ((())()(()))  &\;\;\;  L(L({\ast}){\ast}L({\ast}))  &
22  &\;\;\;  ((())())(())  &\;\;\;  L(L({\ast}){\ast})L({\ast})  \\
23  &\;\;\;  ((()))(()())  &\;\;\;  L(L({\ast}))L({\ast}{\ast})  &
24  &\;\;\;  ((()))(())()  &\;\;\;  L(L({\ast}))L({\ast}){\ast}  \\
25  &\;\;\;  ((()))()(())  &\;\;\;  L(L({\ast})){\ast}L({\ast})  &
26  &\;\;\;  (()((()())))  &\;\;\;  L({\ast}L(L({\ast}{\ast})))  \\
27  &\;\;\;  (()((())()))  &\;\;\;  L({\ast}L(L({\ast}){\ast}))  &
28  &\;\;\;  (()((()))())  &\;\;\;  L({\ast}L(L({\ast})){\ast})  \\
29  &\;\;\;  (()((())))()  &\;\;\;  L({\ast}L(L({\ast}))){\ast}  &
30  &\;\;\;  (()(()(())))  &\;\;\;  L({\ast}L({\ast}L({\ast})))  \\
31  &\;\;\;  (()(())(()))  &\;\;\;  L({\ast}L({\ast})L({\ast}))  &
32  &\;\;\;  (()(()))(())  &\;\;\;  L({\ast}L({\ast}))L({\ast})  \\
33  &\;\;\;  (()()((())))  &\;\;\;  L({\ast}{\ast}L(L({\ast})))  &
34  &\;\;\;  (()())((()))  &\;\;\;  L({\ast}{\ast})L(L({\ast}))  \\
35  &\;\;\;  (())((()()))  &\;\;\;  L({\ast})L(L({\ast}{\ast}))  &
36  &\;\;\;  (())((())())  &\;\;\;  L({\ast})L(L({\ast}){\ast})  \\
37  &\;\;\;  (())((()))()  &\;\;\;  L({\ast})L(L({\ast})){\ast}  &
38  &\;\;\;  (())(()(()))  &\;\;\;  L({\ast})L({\ast}L({\ast}))  \\
39  &\;\;\;  (())(())(())  &\;\;\;  L({\ast})L({\ast})L({\ast})  &
40  &\;\;\;  (())()((()))  &\;\;\;  L({\ast}){\ast}L(L({\ast}))  \\
41  &\;\;\;  ()(((()())))  &\;\;\;  {\ast}L(L(L({\ast}{\ast})))  &
42  &\;\;\;  ()(((())()))  &\;\;\;  {\ast}L(L(L({\ast}){\ast}))  \\
43  &\;\;\;  ()(((()))())  &\;\;\;  {\ast}L(L(L({\ast})){\ast})  &
44  &\;\;\;  ()(((())))()  &\;\;\;  {\ast}L(L(L({\ast}))){\ast}  \\
45  &\;\;\;  ()((()(())))  &\;\;\;  {\ast}L(L({\ast}L({\ast})))  &
46  &\;\;\;  ()((())(()))  &\;\;\;  {\ast}L(L({\ast})L({\ast}))  \\
47  &\;\;\;  ()((()))(())  &\;\;\;  {\ast}L(L({\ast}))L({\ast})  &
48  &\;\;\;  ()(()((())))  &\;\;\;  {\ast}L({\ast}L(L({\ast})))  \\
49  &\;\;\;  ()(())((()))  &\;\;\;  {\ast}L({\ast})L(L({\ast}))  &
50  &\;\;\;  ()()(((())))  &\;\;\;  {\ast}{\ast}L(L(L({\ast})))
\end{array}
$
\bigskip
\caption{Ordered monomial basis of $\mathcal{O}(3,3)$}
\label{basis33}
\end{table}


\begin{table}[ht]
$\left[
\begin{array}{
c@{\quad}c@{\quad}c@{\quad}c@{\quad}c@{\quad}c@{\quad}c@{\quad}c@{\quad}c@{\quad}c@{\quad}c@{\quad}c@{\quad}c@{\quad}c@{\quad}
c@{\quad}c@{\quad}c@{\quad}c@{\quad}c@{\quad}c@{\quad}c@{\quad}c@{\quad}c@{\quad}c@{\quad}c@{\quad}c@{\quad}c@{\quad}c
}
. & . & . & a & . & . & . & a & . & . & . & . & . & . & . & . & . & . & . & . \\[-2pt]
. & . & . & b & . & . & . & . & . & . & . & . & . & . & . & . & a & . & . & . \\[-2pt]
. & . & . & c & . & . & . & . & . & . & . & a & . & . & . & . & b & . & . & . \\[-2pt]
. & . & . & . & . & . & . & . & . & . & . & . & . & . & . & . & c & . & a & . \\[-2pt]
a & . & . & . & a & . & . & b & . & . & . & . & . & . & . & . & . & . & . & . \\[-2pt]
b & . & . & . & . & . & . & . & . & . & . & b & . & . & . & . & . & . & . & . \\[-2pt]
c & . & . & . & . & . & . & . & a & . & . & . & . & . & . & . & . & . & b & . \\[-2pt]
. & . & . & . & b & . & . & c & . & . & . & c & . & . & . & . & . & . & . & . \\[-2pt]
. & . & . & . & . & . & . & . & b & . & . & . & . & . & . & . & . & . & c & . \\[-2pt]
. & . & . & . & c & . & . & . & c & . & . & . & . & . & . & . & . & . & . & . \\[-2pt]
. & . & . & . & . & . & . & d & . & . & . & . & . & . & . & . & . & a & . & . \\[-2pt]
. & a & . & . & . & a & . & . & . & . & . & . & . & . & . & . & . & . & . & . \\[-2pt]
. & b & . & . & . & . & . & . & . & . & . & d & . & . & . & . & . & . & . & . \\[-2pt]
. & c & . & . & . & . & . & . & . & a & . & . & . & . & . & . & . & . & d & . \\[-2pt]
. & . & a & d & . & . & a & . & . & . & . & . & . & . & . & . & . & . & . & . \\[-2pt]
. & . & b & e & . & . & . & . & . & . & . & . & . & . & . & . & d & . & . & . \\[-2pt]
. & . & c & . & . & . & . & . & . & . & a & . & . & . & . & . & e & . & . & . \\[-2pt]
. & . & . & . & d & . & . & e & . & . & . & . & . & . & . & . & . & b & . & . \\[-2pt]
. & . & . & . & . & b & . & . & . & . & . & e & . & . & . & . & . & . & . & . \\[-2pt]
. & . & . & . & . & . & . & . & d & b & . & . & . & . & . & . & . & . & e & . \\[-2pt]
d & . & . & . & . & . & b & . & . & . & . & . & . & . & . & . & . & . & . & . \\[-2pt]
e & . & . & . & . & . & . & . & . & . & b & . & . & . & . & . & . & . & . & . \\[-2pt]
. & . & . & . & e & . & . & . & . & . & . & . & . & . & . & . & . & c & . & . \\[-2pt]
. & . & . & . & . & c & . & . & e & c & . & . & . & . & . & . & . & . & . & . \\[-2pt]
. & . & . & . & . & . & c & . & . & . & c & . & . & . & . & . & . & . & . & . \\[-2pt]
. & . & . & . & . & . & . & f & . & . & . & . & . & . & . & a & . & d & . & . \\[-2pt]
. & . & . & . & . & d & . & . & . & . & . & . & . & . & . & b & . & . & . & . \\[-2pt]
. & . & . & . & . & . & . & . & . & . & . & f & . & . & . & c & . & . & . & . \\[-2pt]
. & . & . & . & . & . & . & . & . & d & . & . & . & . & . & . & . & . & f & . \\[-2pt]
. & . & . & . & . & . & d & . & . & . & . & . & . & . & . & d & . & . & . & . \\[-2pt]
. & d & . & . & . & . & . & . & . & . & . & . & . & . & . & e & . & . & . & . \\[-2pt]
. & e & . & . & . & . & . & . & . & . & d & . & . & . & . & . & . & . & . & . \\[-2pt]
. & . & d & f & . & . & . & . & . & . & . & . & . & . & . & f & . & . & . & . \\[-2pt]
. & . & e & . & . & . & . & . & . & . & . & . & . & . & . & . & f & . & . & . \\[-2pt]
. & . & . & . & f & . & . & . & . & . & . & . & a & . & . & . & . & e & . & . \\[-2pt]
. & . & . & . & . & e & . & . & . & . & . & . & b & . & . & . & . & . & . & . \\[-2pt]
. & . & . & . & . & . & . & . & f & e & . & . & c & . & . & . & . & . & . & . \\[-2pt]
. & . & . & . & . & . & e & . & . & . & . & . & d & . & . & . & . & . & . & . \\[-2pt]
. & . & . & . & . & . & . & . & . & . & e & . & e & . & . & . & . & . & . & . \\[-2pt]
f & . & . & . & . & . & . & . & . & . & . & . & f & . & . & . & . & . & . & . \\[-2pt]
. & . & . & . & . & . & . & . & . & . & . & . & . & . & . & . & . & f & . & a \\[-2pt]
. & . & . & . & . & f & . & . & . & . & . & . & . & a & . & . & . & . & . & b \\[-2pt]
. & . & . & . & . & . & . & . & . & . & . & . & . & b & . & . & . & . & . & c \\[-2pt]
. & . & . & . & . & . & . & . & . & f & . & . & . & c & . & . & . & . & . & . \\[-2pt]
. & . & . & . & . & . & f & . & . & . & . & . & . & . & a & . & . & . & . & d \\[-2pt]
. & . & . & . & . & . & . & . & . & . & . & . & . & d & b & . & . & . & . & e \\[-2pt]
. & . & . & . & . & . & . & . & . & . & f & . & . & e & c & . & . & . & . & . \\[-2pt]
. & . & . & . & . & . & . & . & . & . & . & . & . & . & d & . & . & . & . & f \\[-2pt]
. & f & . & . & . & . & . & . & . & . & . & . & . & f & e & . & . & . & . & . \\[-2pt]
. & . & f & . & . & . & . & . & . & . & . & . & . & . & f & . & . & . & . & .
\end{array}
\right]$
\bigskip
\caption{The transpose matrix $M(R)^t$ for general $R \in \mathcal{O}(2,2)$}
\label{conmat22}
\end{table}


The transpose $M(R)^t$ of the $20 \times 50$ matrix $M(R)$ appears in
Table \ref{conmat22}.
Our goal is to understand how the rank of this matrix depends on the values of
the parameters $a, b, c, d, e, f$.
Our results depend on computational commutative algebra
over the polynomial ring $\mathbb{F}[a,b,c,d,e,f]$.
In particular, we compute Gr\"obner bases for determinantal ideals
of $M(R)^t$.
We use the degree-lexicographical monomial order
$a \prec b \prec c \prec d \prec e \prec f$.
Our computations depend on the following result from linear algebra,
and its generalization to matrices over polynomial rings.

\begin{theorem}
Let $M$ be an $m \times n$ matrix with coefficients in the field $\mathbb{F}$.
For $0 \le r \le \min(m,n)$ we have $\mathrm{rank}(M) = r$ if and only if
\begin{enumerate}
\item[(i)]
some $r \times r$ submatrix of $M$ has nonzero determinant, and
\item[(ii)]
every $(r{+}1) \times (r{+}1)$ submatrix of $M$ has determinant 0.
\end{enumerate}
(If $r = \min(m,n)$ then condition (ii) is vacuously true.)
\end{theorem}

\begin{definition}
Let $M$ be an $m \times n$ matrix with coefficients
in the polynomial ring $P = \mathbb{F}[x_1,\dots,x_k]$.
Let $I(M,r)$ be the ideal in $P$ generated by the determinants
of all $r \times r$ submatrices of $M$;
this is the $r$-th \emph{determinantal ideal} of $M$.
Let $Z(M,r) \subseteq \mathbb{F}^k$ be the zero set of $I(M,r)$;
that is, the $k$-tuples for which every polynomial in $I(M,r)$ evaluates to 0.
\end{definition}

\begin{proposition}
For $0 \le r < s \le \min(m,n)$ we have $I(M,r) \supseteq I(M,s)$
and hence $Z(M,r) \subseteq Z(M,s)$.
For $X = (a_1,\dots,a_k) \in \mathbb{F}^k$ let $M_X$ denote the matrix with entries in $\mathbb{F}$
obtained by substituting $x_1 = a_1, \dots, x_k = a_k$ in $M$.
Then $M_X$ has rank $r$ if and only if $X \in Z(M,r{+}1)$ but $X \notin Z(M,r)$.
\end{proposition}


\subsection{Case 1: $a = 1$ and $b, c, d, e, f$ are free}

\subsubsection{Computing the lower right block}

We substitute $a = 1$ in the matrix $M(R)^t$ of Table \ref{conmat22}
and compute the partial Smith form:
\[
\left[
\begin{matrix}
I_{16} & O_{16,4} \\
O_{34,16} & B
\end{matrix}
\right]
\]
The lower right block $B$ has size $34 \times 4$; see Table \ref{case1matrixB}.
From this it is clear that $M(R)$ has minimal rank 16,
and that $\text{rank} \, M(R) = 16 + r$ where $r = \text{rank} \, B$.
Thus $B$ is a matrix over $\mathbb{F}[b,c,d,e,f]$
and we need to determine how the rank of $B$ depends on the parameters $b, c, d, e, f$.

\begin{table}[ht]
$\left[
\begin{array}{c@{\qquad}c@{\qquad}c@{\qquad}c}
b & . & . & -b^2 c+c^2+c \\
d & . & . & -d b+e \\
. & b & . & -b^2 e+e c \\
c d & c b & . & 2 b^2 c d+d c b-e c b \\
-d & . & b & -d b \\
-e & . & c b & -b^2 e-d c b-e b \\
e & . & . & -c d \\
e c & c^2+c & . & b^2 c e+b c^2 d+e c b \\
. & . & c^2+c & -e c b-c^2 d \\
. & -b & . & -b^2 d+c d \\
. & d & . & d^2 b-f b \\
. & . & . & -f b^2+d^2 c \\
. & c d & . & b c d^2-f c b \\
. & . & d & d^3-f d \\
. & -d & . & e d^2-f e \\
. & -e & c d & -e b d-d^2 c \\
. & . & -d & f d^2+d^2-f^2-f \\
. & . & -e & f b+e d \\
c b & . & . & b^3 c+b^2 c-c^2 b \\
-f b & e & . & e b d \\
. & e c & . & b^2 c f+b c d e+f c b+e c d \\
-f d & . & e & e d^2 \\
-f e & . & e c & -b e^2-e c d+e^2 d \\
-f^2-f & . & . & f e d-f b \\
. & . & -b & 2 d b-e \\
c^2+c & . & . & c^2 b^2+c^2 b \\
. & -f b & . & -b^2 d f+f c d \\
. & . & . & . \\
-b & . & . & -b^3-b^2+c b \\
. & -f d & -f b & -2 b d^2 f+f e d \\
. & -f e & . & -b d e f-c d^2 f-b e f-f c d \\
. & . & -f d & -d^3 f+d f^2 \\
. & -f^2-f & -f e & -b d f^2-d^2 e f \\
. & . & -f^2-f & -f^2 d^2+f d
\end{array}
\right]$
\bigskip
\caption{Case 1: lower right block $B$}
\label{case1matrixB}
\end{table}

\subsubsection{The first determinantal ideal}

Clearly $I(B,1)$ is the ideal generated by the 47 distinct nonzero entries of $B$.
The Gr\"obner basis is
$b$, $d$, $e$, $c(c+1)$, $f(f+1)$.
The zero set (including the assumption $a = 1$)
consists of the rows of the array
\[
\begin{array}{rrrrrr}
a & b &  c & d & e &  f \\ \midrule
1 & 0 & -1 & 0 & 0 & -1 \\
1 & 0 & -1 & 0 & 0 &  0 \\
1 & 0 &  0 & 0 & 0 & -1 \\
1 & 0 &  0 & 0 & 0 &  0
\end{array}
\]
These are the values of the parameters which make $B = O$,
and hence for these values $M(R)$ has rank 16.

\begin{proposition}
\label{case1rank16}
If $a = 1$ then $M(R)$ has rank 16 if and only if
the corresponding operator identity is one of the following:
\begin{align*}
&
L^2(xy) = L^2(x)y + xL^2(y),
\qquad
L^2(xy) = L^2(x)y,
\\
&
L^2(xy) = xL^2(y),
\qquad
L^2(xy) = 0.
\end{align*}
\end{proposition}

\subsubsection{The second determinantal ideal}

At this point we begin to rely heavily on results obtained using the
computer algebra system Maple.
The matrix $B$ has $\binom{34}{2}\binom{4}{2} = 3366$ subdeterminants of size 2,
consisting of 817 distinct polynomials of degrees 2 to 6.
The Gr\"obner basis of $I(B,2)$ has size 15,
but its zero set is the same as that of $I(B,1)$:
these two ideals have the same radical.
Hence there are no values of the parameters for which $M(R)$ has rank 17.

\subsubsection{The third determinantal ideal}

The matrix $B$ has $\binom{34}{3}\binom{4}{3} = 23936$ subdeterminants of size 3,
consisting of 6921 distinct polynomials of degrees 3 to 8.
The Gr\"obner basis of $I(B,3)$ has size 35,
but its zero set is the same as that of $I(B,2)$.
Hence there are no values of the parameters for which $M(R)$ has rank 18.

\subsubsection{The fourth determinantal ideal}
\label{case1fourth}

The matrix $B$ has $\binom{34}{4}\binom{4}{4} = 46376$ subdeterminants of size 4,
consisting of 20363 distinct polynomials of degrees 5 to 10.
The Gr\"obner basis of $I(B,4)$ consists of 93 polynomials;
see Table \ref{gb93} at the end of this paper.
The first element of the Gr\"obner basis is $b e^2 (d-b)$.
Hence we may split the computation of the zero set into three subcases:
$b = 0$, $e = 0$ and $d = b$.

\textsf{Subcase $b = 0$}:
We substitute $b = 0$ into the Gr\"obner basis and obtain a set of 40 nonzero polynomials.
Using Maple to solve this system of polynomial equations
(and including the assumptions $a = 1$, $b = 0$)
we obtain five solutions:
\[
\begin{array}{rrrrrr}
a & b &  c & d & e &  f \\ \midrule
1 & 0 & -1 & 0 & 0 & -1 \\
1 & 0 & -1 & 0 & 0 &  0 \\
1 & 0 &  0 & 0 & 0 &  0 \\
1 & 0 &  0 & 1 & 0 &  1 \\
1 & 0 &  0 & d & 0 & -d{-}1
\end{array}
\]
In the last solution $d$ is a free parameter.
Excluding the solutions giving rank 16 obtained from the first determinantal ideal,
we are left with the fourth solution and case $d \ne 0$ of the fifth solution,
giving two new solutions for which $M(R)$ has rank 19.
See the first pair of identities in Proposition \ref{case1rank19}.

\textsf{Subcase $e = 0$}:
We substitute $e = 0$ into the Gr\"obner basis and obtain a set of 55 nonzero polynomials.
Using Maple (and including the assumptions $a = 1$, $e = 0$) we obtain six solutions:
\[
\begin{array}{rrrrrr}
a & b &    c &    d & e &  f \\ \midrule
1 & 0 &   -1 &    0 & 0 & -1 \\
1 & 0 &    0 &    0 & 0 &  0 \\
1 & 0 &    0 &    1 & 0 &  1 \\
1 & 0 &    0 & -f{-}1 & 0 &  f \\
1 & 1 &    1 &    0 & 0 &  0 \\
1 & b & -b{-}1 &    0 & 0 &  0
\end{array}
\]
In the fourth and sixth solutions $f$ and $b$ are free parameters.
The first two solutions come from the first determinantal ideal (giving rank 16).
The third solution was obtained for $b = 0$.
The fourth solution coincides with the last solution for $b = 0$
since $d = -f{-}1$ if and only if $f = -d{-}1$.
The fifth solution is new (giving rank 19).
The sixth solution is new except for the special case $b = 0$ which comes from
the first determinantal ideal (giving rank 16).
This gives two new solutions for which $M(R)$ has rank 19.
See the second pair of identities in Proposition \ref{case1rank19}.

\textsf{Subcase $d = b$}:
We substitute $d = b$ into the Gr\"obner basis and obtain a set of 63 nonzero polynomials.
Using Maple (and including the assumptions $a = 1$, $d = b$)
we obtain six solutions:
\[
\begin{array}{rrrrrr}
a &  b &  c &  d & e &  f \\ \midrule
1 & -2 &  1 & -2 & 2 &  1 \\
1 & -1 &  0 & -1 & 1 &  0 \\
1 &  0 & -1 &  0 & 0 & -1 \\
1 &  0 & -1 &  0 & 0 &  0 \\
1 &  0 &  0 &  0 & 0 & -1 \\
1 &  0 &  0 &  0 & 0 &  0
\end{array}
\]
The last four come from the first determinantal ideal (giving rank 16).
The first two are new (giving rank 19).
See the third pair of identities in Proposition \ref{case1rank19}.

\begin{proposition}
\label{case1rank19}
If $a = 1$ then $M(R)$ has rank 19 if and only if
the corresponding operator identity is one of the following:
\begin{align*}
&
L^2(xy) + L(xL(y)) + xL^2(y) = 0,
\\
&
L^2(xy) + d L(xL(y)) = (d{+}1) xL^2(y) \quad (d \ne 0),
\\
&
L^2(xy) + L(L(x)y) + L^2(x)y = 0,
\\
&
L^2(xy) + b L(L(x)y) = (b{+}1) L^2(x)y \quad (b \ne 0),
\\
&
L^2(xy) + L^2(x)y + xL^2(y)
=
2 \big[ L(L(x)y) + L(xL(y)) - L(x)L(y) \big],
\\
&
L^2(xy) + L(x)L(y) = L(L(x)y) + L(xL(y)).
\end{align*}
\end{proposition}

\begin{remark}
If $a = 1$ then $M(R)$ has full rank 20 except for the values of $b, c, d, e, f$
given in Propositions \ref{case1rank16} and \ref{case1rank19}.
\end{remark}


\subsection{Case 2: $a = 0$, $b = 1$ and $c, d, e, f$ are free}

We substitute $a = 0$, $b = 1$ in $M(R)^t$ and compute the partial Smith form:
\[
\left[
\begin{matrix}
I_{19} & O_{19,1} \\
O_{31,19} & B
\end{matrix}
\right]
\]
The lower right block $B$ is a vector of size 31 containing 22 distinct nonzero entries:
\begin{align*}
&
c(d {+} e), \;\;
c(cd {-} e), \;\;
c(de {-} f), \;\;
c(c^2d {-} 2ce {-} e), \;\;
c(cd^2 {-} de {+} f), \;\;
c^2(de {-} f), \;\;
d^2c,
\\
&
d^2(e {+} 1), \;\;
e^2(e {+} 1), \;\;
f(cde^2 {-} cef {-} e^3 {-} d), \;\;
cdf, \;\;
cf(cd {-} ce {-} e), \;\;
ed(e {+} 1),
\\
&
fd(e {+} 1), \;\;
fed, \;\;
fe(c^2d {-} 2ce {+} 1), \;\;
fe(cd^2 {-} de {+} f), \;\;
f({-}1 {+} e)(e {+} 1), \;\;
f^2e(cd {-} e),
\\
&
{-}c^2(c {+} 1), \;\;
{-}d({-}e^2 {+} d), \;\;
cde {+} f.
\end{align*}
The ideal generated by these polynomials has this Gr\"obner basis:
\[
f, \;\; c(d {+} e), \;\; d(d {+} e), \;\; c^2(c {+} 1), \;\; dc(c {+} 1), \;\; d^2c, \;\; d^2(d {-} 1), \;\; e^2(e {+} 1).
\]
The zero set of this ideal (including $a=0$, $b=1$) gives four solutions:
\[
\begin{array}{rrrrrr}
a & b &  c & d &  e & f \\ \midrule
0 & 1 & -1 & 0 &  0 & 0 \\
0 & 1 &  0 & 0 & -1 & 0 \\
0 & 1 &  0 & 0 &  0 & 0 \\
0 & 1 &  0 & 1 & -1 & 0
\end{array}
\]

\begin{proposition}
\label{case2rank19}
If $a = 0$, $b = 1$ then the matrix $M(R)$ has rank 19 if and only if
the corresponding operator identity is one of the following:
\begin{align*}
&
L(L(x)y) = L^2(x)y,
\\
&
L(L(x)y) = L(x)L(y),
\\
&
L(L(x)y) = 0,
\\
&
L(L(x)y) + L(xL(y)) = L(x)L(y).
\end{align*}
For all other values of $c, d, e, f$ the matrix $M(R)$ has full rank 20.
\end{proposition}


\subsection{Case 3: $a = b = 0$, $c = 1$ and $d, e, f$ are free}

We substitute $a = b = 0$, $c = 1$ in $M(R)^t$ and compute the partial Smith form:
\[
\left[
\begin{matrix}
I_{16} & O_{16,4} \\
O_{34,16} & B
\end{matrix}
\right]
\]
The lower right block $B$ has size $34 \times 4$ but we delete the 10 zero rows;
see Table \ref{case3matrixB}.

\begin{table}[ht]
$\left[
\begin{array}{c@{\qquad}c@{\qquad}c@{\qquad}c}
. & . & . & d \\
-d & . & . & -e \\
. & . & . & e \\
d & . & . & . \\
e d & . & . & -f \\
-e d & -d & . & . \\
. & d & . & . \\
. & . & -d & -f d \\
. & . & . & -f e \\
. & . & d & . \\
. & . & . & -f^2 \\
e^2-f & . & . & . \\
-e^2 & -e & . & . \\
-f d & -e d & -e & . \\
-f e & -e^2 & e & . \\
-f^2 & -f e & . & . \\
f e & . & . & . \\
-f e & -f & . & . \\
. & . & -f & . \\
. & -f d & . & . \\
. & . & . & -d \\
. & f e d & -f d & . \\
. & e^2 f-f^2 & -f e & . \\
. & f^2 e & -f^2 & .
\end{array}
\right]$
\bigskip
\caption{Case 3: lower right block $B$}
\label{case3matrixB}
\end{table}

Clearly the first determinantal ideal has Gr\"obner basis $d, e, f$.
The second determinantal ideal is generated by 102 polynomials of degrees 2 to 5,
and has Gr\"obner basis
$d^2$, $ed$, $fd$, $e^2$, $fe$, $f^2$.
The third determinantal ideal is generated by 316 polynomials of degrees 3 to 7,
and has Gr\"obner basis
$d^3$, $ed^2$, $fd^2$, $de^2$, $fed$, $df^2$, $e^3$, $e^2f$, $f^2e$, $f^3$.
The fourth determinantal ideal is generated by 643 polynomials of degrees 4 to 8,
and has Gr\"obner basis
$d^4$, $d^3e$, $d^3f$, $e^2d^2$, $d^2ef$, $d^2f^2$, $de^3$, $de^2f$, $def^2$, $df^3$, $e^2f^2$, $ef^3$, $f^4$, $e^2{-}f)e^3$, $e^4f$.
It is obvious that all four determinantal ideals all have the same zero set,
namely the single point $(a,b,c,d,e,f) = (0,0,1,0,0,0)$.

\begin{proposition}
\label{case3rank16}
If $a = b = 0$, $c = 1$ then the matrix $M(R)$ has rank 16 if and only if
$d = e = f = 0$ which corresponds to the operator identity
\[
L^2(x)y = 0.
\]
For all other values of $d, e, f$ the matrix $M(R)$ has full rank 20.
\end{proposition}


\subsection{Case 4: $a = b = c = 0$, $d = 1$ and $e, f$ are free}

We substitute $a = b = c = 0$, $d = 1$ in $M(R)^t$ and compute the partial Smith form:
\[
\left[
\begin{matrix}
I_{19} & O_{19,1} \\
O_{31,19} & B
\end{matrix}
\right]
\]
The lower right block $B$ is a vector of size 31 containing 5 distinct nonzero entries:
\[
e^2(e + 1), \quad fe, \quad -fe, \quad -f^2(f + 1), \quad -fe(1 + 2f).
\]
The ideal generated by these polynomials has this Gr\"obner basis:
\[
fe, \quad e^2(e + 1), \quad f^2(f + 1).
\]
The zero set of this ideal (including $a=b=c=0$, $d=1$) gives three solutions:
\[
\begin{array}{rrrrrr}
a & b & c & d &  e & f \\ \midrule
0 & 0 & 0 & 1 & -1 &  0 \\
0 & 0 & 0 & 1 &  0 & -1 \\
0 & 0 & 0 & 1 &  0 &  0
\end{array}
\]

\begin{proposition}
\label{case4rank19}
If $a = b = c = 0$, $d = 1$ then the matrix $M(R)$ has rank 19 if and only if
the corresponding operator identity is one of the following:
\begin{align*}
&
L(xL(y)) = L(x)L(y),
\\
&
L(xL(y)) = xL^2(y),
\\
&
L(xL(y)) = 0.
\end{align*}
For all other values of $e, f$ the matrix $M(R)$ has full rank 20.
\end{proposition}


\subsection{Case 5: $a = b = c = d = 0$, $e = 1$ and $f$ is free}

We have only one indeterminate so we can compute the ordinary Smith form.
We substitute $a = b = c = d = 0$, $e = 1$ in $M(R)^t$, compute the Smith form,
and find the diagonal entries 1 (19 times) and $f$.
For $f = 0$ we obtain the operator identity $L(x)L(y) = 0$.

\begin{proposition}
\label{case5rank19}
If $a = b = c = d = 0$, $e = 1$ then the matrix $M(R)$ has rank 19 if and only if
$f = 0$ for which the corresponding operator identity is
\[
L(x)L(y) = 0.
\]
For all other values of $f$ the matrix $M(R)$ has full rank 20.
\end{proposition}


\subsection{Case 6: $a = b = c = d = e = 0$ and $f = 1$}

There are no indeterminates so we can compute the row canonical form.
Omitting zero rows, we obtain $I_{16}$.

\begin{proposition}
\label{case6rank16}
If $a = b = c = d = e = 0$, $f = 1$ then the matrix $M(R)$ has rank 16
and the corresponding operator identity is
\[
xL^2(y).
\]
\end{proposition}

We now collect the results of Propositions
\ref{case1rank16},
\ref{case1rank19},
\ref{case2rank19},
\ref{case3rank16},
\ref{case4rank19},
\ref{case5rank19},
\ref{case6rank16}.
This is our new classification of operator identities of degree 2 and multiplicity 2.

\begin{theorem}
\label{theorem22}
The matrix $M(R)^t$ of Table \ref{conmat22} has rank 16 if and only if
the parameters $a, b, c, d, e, f$ correspond to one of the following operator identities
which are obtained from Theorem \ref{theorem21} by replacing $L$ by $L^2$:
\begin{align}
& L^2(xy) = L^2(x)y + xL^2(y)
\\
& L^2(xy) = L^2(x)y
\\
& L^2(xy) = xL^2(y)
\\
& L^2(xy) = 0
\\
& L^2(x)y = 0
\\
& xL^2(y) = 0
\end{align}
The matrix $M(R)^t$ of Table \ref{conmat22} has rank 19 if and only if
the parameters $a, b, c, d, e, f$ correspond to one of the following operator identities:
\begin{alignat}{2}
& L^2(xy) + L(xL(y)) + xL^2(y) = 0 &\quad & \text{New identity A (right)}
\\
& L^2(xy) + d L(xL(y)) = (d{+}1) xL^2(y) \quad (d \ne 0) &\quad & \text{New identity B (right)}
\\
& L^2(xy) + L(L(x)y) + L^2(x)y = 0 &\quad & \text{New identity A (left)}
\\
& L^2(xy) + b L(L(x)y) = (b{+}1) L^2(x)y \quad (b \ne 0) &\quad & \text{New identity B (left)}
\\
& L^2(xy) + L^2(x)y + xL^2(y) &\quad & \text{New identity C}
\\
\notag
&=
2 \big[ L(L(x)y) + L(xL(y)) - L(x)L(y) \big]
\\
& L^2(xy) + L(x)L(y) = L(L(x)y) + L(xL(y)) &\quad & \text{Nijenhuis}
\\
& L(L(x)y) = L^2(x)y &\quad & P_1
\\
& L(L(x)y) = L(x)L(y) &\quad & \text{Left average}
\\
& L(L(x)y) = 0 &\quad & P_2
\\
& L(L(x)y) + L(xL(y)) = L(x)L(y) &\quad & \text{Rota-Baxter}
\\
& L(xL(y)) = L(x)L(y) &\quad & \text{Right average}
\\
& L(xL(y)) = xL^2(y) &\quad & P_3
\\
& L(xL(y)) = 0 &\quad & P_4
\\
& L(x)L(y) = 0 &\quad & P_5
\end{alignat}
For all other values of the parameters, the matrix $M(R)^t$ has full rank 20.
\end{theorem}

\begin{remark}
Our classification includes three new identities $A$, $B$ and $C$;
identities $A$ and $B$ occur in two opposite (left and right) versions.
The right side of identity $C$ is equivalent to (twice)
the Rota-Baxter identity.
It would be very interesting to find applications of these identities.
\end{remark}

\begin{remark}
Some of the identities of Theorem \ref{theorem22} can be obtained
from identities of Theorem \ref{theorem21} by partial composition with $L$.
In particular,
the left average identity comes from applying $- \circ_1 L$ to $L(xy) = xL(y)$, and
the right average identity comes from applying $- \circ_2 L$ to $L(xy) = L(x)y$.
Similar comments apply to the identities denoted $P_1, \dots, P_5$.
\end{remark}

\begin{table}
\footnotesize
\begin{align*}
&
be^2(d{-}b), \;\;
be^2(b{+}e), \;\;
e^2(d{-}b)(d{+}b), \;\;
e^2(b^2{+}de), \;\;
e^2(e{-}b)(e{+}b), \;\;
b^3(2bd{+}be{+}cd{-}e),
\\
&
{-}b^3(2bd{+}be{-}ce{-}e), \;\;
b^2(2b^2f{+}2bd^2{-}de{+}e^2), \;\;
{-}b^2(b^2f{-}bde{+}e^2), \;\;
b^3f(d{-}b), \;\;
b^2(b^2f{+}be^2{+}e^2),
\\
&
b^3f(b{+}e), \;\;
b^2(2b^2f{+}2bf^2{+}2bf{-}de{-}e^2), \;\;
{-}b^2(3b^2d{+}2b^2e{-}c^2d{-}be{-}cd),
\\
&
b^2(2b^2d{+}b^2e{+}c^2e{-}be{+}ce), \;\;
b^2(cd^2{-}b^2f), \;\;
b^2(b^2f{+}cde), \;\;
b^2(ce^2{-}b^2f), \;\;
b(2b^3f{+}2bd^3{+}be^2{-}d^2e),
\\
&
{-}b^2(b^2f{-}d^2e{+}e^2), \;\;
{-}fb^2(b{-}d)(b{+}d), \;\;
fb^2(b^2{+}de), \;\;
b(2b^3f{+}2bdf^2{+}2bdf{-}be^2{-}d^2e),
\\
&
fb^2(e{-}b)(e{+}b), \;\;
{-}fb^2(b^2{-}ef{-}e), \;\;
b(2b^3f{+}2c^2d^2{-}bde{-}be^2{+}2cd^2), \;\;
{-}b(b^3f{-}c^2de{-}cde),
\\
&
b(b^3f{+}c^2e^2{+}ce^2), \;\;
b(cd^3{-}b^3f), \;\;
b(b^3f{+}cd^2e), \;\;
{-}b(b^3f{-}d^3e{+}be^2), \;\;
fb(b^3{+}d^2e),
\\
&
d^2(bd^2{+}2bdf{+}bf^2{+}bf{-}de), \;\;
{-}fb(b^3{-}def{-}de), \;\;
fb(b^3{+}e^2f{+}e^2),
\\
&
2b^4f{+}2c^2d^3{-}b^2e^2{-}bd^2e{+}2cd^3, \;\;
{-}b^4f{+}c^2d^2e{+}cd^2e, \;\;
b^4f{+}c^2de^2{+}cde^2, \;\;
{-}b^4f{+}c^2e^3{+}ce^3,
\\
&
{-}b^4f{+}cd^4, \;\;
b^4f{+}cd^3e, \;\;
d^3(2bd{+}bf{+}de{-}e), \;\;
d^3f(b{+}e), \;\;
{-}fd^2(bd{-}ef{-}e), \;\;
f(b^4{+}de^2f{+}de^2),
\\
&
{-}f(b^4{-}e^3f{-}e^3), \;\;
b^3(b{+}c{+}1)(b^2{-}c), \;\;
b^3(b^2d{+}4bd{+}be{-}2e), \;\;
b^3(b^2e{-}4bd{-}be{+}2e), \;\;
b^4f(b{+}2),
\\
&
b^4f(c{-}1), \;\;
b^3(b{+}c{+}1)(b^2{-}bc{+}c^2{-}c), \;\;
b^3f(c^2{+}b{+}c), \;\;
b^2(b{+}c{+}1)({-}c{+}b)(b^2{-}c^2{-}c),
\\
&
fb^2(c^2f{-}b^2{+}c^2{+}cf{+}c), \;\;
fb^2(f^3{-}b^2{+}2f^2{+}f), \;\;
b(c^4d{+}2b^3d{+}b^3e{+}2c^3d{-}b^2e{+}c^2d),
\\
&
b(c^4e{-}2b^3d{-}b^3e{+}2c^3e{+}b^2e{+}c^2e), \;\; 
fb(c^2df{-}b^3{+}c^2d{+}cdf{+}cd), \;\;
fb(c^2ef{+}b^3{+}c^2e{+}cef{+}ce),
\\
&
d^3(bd^2{+}2bd{-}bf{-}e), \;\; 
bd^3f(d{+}2), \;\; 
bdf(f^3{-}d^2{+}2f^2{+}f), \;\;
fb(ef^3{+}b^3{+}2ef^2{+}ef),
\\
&
c^4d^2{-}b^4f{+}2c^3d^2{+}c^2d^2, \;\; 
c^4de{+}b^4f{+}2c^3de{+}c^2de,
\\
&
c^4e^2{-}b^4f{+}2c^3e^2{+}c^2e^2, \;\; 
f(c^2d^2f{-}b^4{+}c^2d^2{+}cd^2f{+}cd^2), \;\;
f(c^2def{+}b^4{+}c^2de{+}cdef{+}cde),
\\
&
f(c^2e^2f{-}b^4{+}c^2e^2{+}ce^2f{+}ce^2), \;\; 
f(cd^3f{+}b^4{+}cd^3), \;\;
d^3(d{+}f{+}1)(d^2{-}f),
\\
&
d^3(d{+}f{+}1)(d^2{-}df{+}f^2{-}f), \;\; 
d^2(d{+}f{+}1)({-}f{+}d)(d^2{-}f^2{-}f), \;\;
fd(ef^3{+}bd^2{+}2ef^2{+}ef),
\\
&
f(e^2f^3{-}b^4{+}2e^2f^2{+}e^2f), \;\;
b(b{-}c)(b{+}c{+}1)(b^4{+}b^2c^2{+}c^4{-}b^3{+}b^2c{+}bc^2{+}2c^3{+}bc{+}c^2),
\\
&
fb(c^4f{+}c^4{+}2c^3f{+}b^3{+}2c^3{+}c^2f{+}c^2), \;\;
fb(c^2f^3{+}2c^2f^2{+}cf^3{+}b^3{+}c^2f{+}2cf^2{+}cf),
\\
&
fb(f^5{+}3f^4{+}d^3{+}3f^3{+}f^2), \;\;
c^6d{+}3c^5d{-}2b^4d{-}b^4e{+}3c^4d{+}b^3e{+}c^3d,
\\
&
c^6e{+}3c^5e{+}2b^4d{+}b^4e{+}3c^4e{-}b^3e{+}c^3e, \;\;
f(c^4df{+}c^4d{+}2c^3df{+}b^4{+}2c^3d{+}c^2df{+}c^2d),
\\
&
f(c^4ef{+}c^4e{+}2c^3ef{-}b^4{+}2c^3e{+}c^2ef{+}c^2e), \;\;
f(c^2df^3{+}2c^2df^2{+}cdf^3{+}b^4{+}c^2df{+}2cdf^2{+}cdf),
\\
&
f(c^2ef^3{+}2c^2ef^2{+}cef^3{-}b^4{+}c^2ef{+}2cef^2{+}cef),
\\
&
d(f{+}d{+}1)(f{-}d)(d^4{+}d^2f^2{+}f^4{-}d^3{+}d^2f{+}df^2{+}2f^3{+}df{+}f^2), \;\; 
f(ef^5{+}3ef^4{-}bd^3{+}3ef^3{+}ef^2),
\\
&
(c{+}b{+}1)(c{-}b)
(b^6{+}b^4c^2{+}b^2c^4{+}c^6{+}3b^5{+}b^4c{+}2b^3c^2{+}2b^2c^3{+}bc^4{+}3c^5{-}b^4{+}2b^3c{+}2b^2c^2{+}2bc^3
\\
&\qquad\qquad\qquad\;
{+}3c^4{+}b^2c{+}bc^2{+}c^3),
\\
&
f(c^6f{+}c^6{+}3c^5f{+}3c^5{+}3c^4f{-}b^4{+}3c^4{+}c^3f{+}c^3), \;\;
\\
&
f(c^4f^3{+}2c^4f^2{+}2c^3f^3{+}c^4f{+}4c^3f^2{+}c^2f^3{-}b^4{+}2c^3f{+}2c^2f^2{+}c^2f),
\\
&
f(c^2f^5{+}3c^2f^4{+}cf^5{+}3c^2f^3{+}3cf^4{-}b^4{+}c^2f^2{+}3cf^3{+}cf^2),
\\
&
(f{-}d)(f{+}d{+}1)(d^6{+}d^4f^2{+}d^2f^4{+}f^6{+}3d^5{+}d^4f{+}2d^3f^2{+}2d^2f^3{+}df^4{+}3f^5{-}d^4{+}2d^3f{+}2d^2f^2
\\
&\qquad\qquad\qquad\;\;
{+}2df^3{+}3f^4{+}d^2f{+}df^2{+}f^3).
\end{align*}
\caption{%
Degree 2, multiplicity 2, case 1 ($a = 1$, other parameters free):
The deglex Gr\"obner basis for the fourth determinantal ideal
of the lower right block of the partial Smith form of 
the matrix of consequences; see \S\ref{case1fourth}.
}
\label{gb93}
\end{table}



\end{document}